\documentclass[11pt,oneside,reqno]{amsart}
\usepackage{amssymb}
\usepackage{color}
\usepackage{tikz-cd}
\usetikzlibrary{decorations.markings}

\usepackage{xcolor}
\definecolor{deepgreen}{cmyk}{0.99998,0,1,0}

\usepackage{hyperref}
\hypersetup{hypertex=true,
	colorlinks=true,
	linkcolor=blue,
	anchorcolor=blue,
	citecolor=blue}
\usepackage{comment}
\usepackage{mathrsfs}

\allowdisplaybreaks[4]

\usepackage{etoolbox}

\usepackage{bm}
\usepackage{bbm}

\usepackage[new]{old-arrows}

\theoremstyle{definition}
\newtheorem{defi}{$\mathbf{Definition}$}[section]
\newtheorem*{pro}{$\mathbf{Proof}$}

\theoremstyle{plain}
\newtheorem{theo}[defi]{$\mathbf{Theorem}$}

\newtheorem{coro}[defi]{$\mathbf{Corollary}$}

\newtheorem{prop}[defi]{$\mathbf{Proposition}$}

\numberwithin{equation}{section}

\newcommand{\pa}{\partial}

\newcommand{\lv}{\left\vert}
\newcommand{\rv}{\right\vert}
\newcommand{\lV}{\left\Vert}
\newcommand{\rV}{\right\Vert}

\newcommand{\Bv}{\Big\vert}

\newcommand{\bV}{\big\Vert}

\newcommand{\bbV}{\bigg\Vert}

\newcommand{\cref}{\S\ \ref}
\newcommand{\Cref}{\S\ \ref}

\newcommand{\Sc}{\mathrm{Sc}}

\DeclareFontFamily{U}{mathx}{\hyphenchar\font45}
\DeclareFontShape{U}{mathx}{m}{n}{
	<5> <6> <7> <8> <9> <10>
	<10.95> <12> <14.4> <17.28> <20.74> <24.88>
	mathx10
}{}
\DeclareSymbolFont{mathx}{U}{mathx}{m}{n}
\DeclareFontSubstitution{U}{mathx}{m}{n}
\DeclareMathAccent{\widecheck}{0}{mathx}{"71}
\DeclareMathAccent{\wideparen}{0}{mathx}{"75}

\newcommand{\sgn}{\mathrm{sgn}}
\newcommand{\R}{\mathbb{R}}

\newcommand{\supp}{\mathrm{supp}}

\newcommand{\DM}{{\widetilde{D}}}

\newcommand{\interior}[1]{%
	{\kern0pt#1}^{\mathrm{\,o}}%
}
\definecolor{pink}{RGB}{249,164,186}
\definecolor{grassgreen}{RGB}{128,255,0}

\numberwithin{equation}{section}
\title[]{Gromov's Simplicial Volume Vanishing Conjecture for Positive Scalar Curvature}
\author{Qiaochu Ma}
\address[Qiaochu Ma]{Department of Mathematics, Texas A\&M University }
\email{qiaochu@tamu.edu}
\author{Jinmin Wang}
\address[Jinmin Wang]{State Key Laboratory of Mathematical Sciences, Academy of Mathematics and Systems Science, Chinese Academy of Sciences}
\email{jinmin@amss.ac.cn}
\author{Zhizhang Xie}
\address[Zhizhang Xie]{ Department of Mathematics, Texas A\&M University }
\email{xie@tamu.edu}
\author{Guoliang Yu}
\address[Guoliang Yu]{ Department of Mathematics, Texas A\&M University }
\email{guoliangyu@tamu.edu}
\author{Bo Zhu}
\address[Bo Zhu]{Yau Mathematical Sciences Center, Tsinghua University}
\email{zhub@tsinghua.edu.cn}
\date{}

\makeatletter
\newcommand*{\rom}[1]{{\mathscr{B}_{\varepsilon}}andafter\@slowromancap\romannumeral #1@}
\makeatother

\begin{document} 
	\clearpage
	
	\begin{abstract}
		%	\fontsize{10pt}{11pt}\selectfont 
		In this paper, we prove Gromov's simplicial volume vanishing conjecture for closed manifolds with spin universal cover. More precisely, we show that if a closed oriented manifold admits a metric of nonnegative scalar curvature and its universal cover is spin, then its simplicial volume vanishes. In particular, a closed oriented aspherical manifold with nonzero simplicial volume admits no metric of nonnegative scalar curvature.

	\end{abstract}
	\maketitle
	\section{Introduction}

	\subsection{Background}\label{back}

	Scalar curvature plays a fundamental role in differential geometry and in Einstein's general theory of relativity, see Lichnerowicz \cite{MR156292}, Schoen-Yau \cite{MR541332,MR526976}, Gromov-Lawson \cite{MR569070,MR577131}, and Witten \cite{MR626707}. The study of manifolds with controlled scalar curvature lies at the interface between flexibility and rigidity. Although scalar curvature is the weakest curvature invariant, it nevertheless imposes subtle topological constraints. A central theme in scalar curvature geometry is to understand how positive scalar curvature restricts the topological complexity of a manifold.
	
	One of the fundamental results in this direction is the Lichnerowicz vanishing theorem \cite{MR156292}, which detects topological complexity through a characteristic number called $\widehat{A}$-genus,
	\begin{equation}\label{l1.1}
		\parbox{0.8\textwidth}{\centering\textit{if a closed spin manifold admits a metric of positive scalar curvature, then its  $\widehat A$-genus vanishes.}}
	\end{equation}
	
	Gromov's \emph{simplicial volume} is another important invariant that provides a quantitative measure of the topological complexity of a manifold. It is defined as the infimum of the sum of the absolute values of the coefficients in a singular cycle representing the fundamental class. In \cite[\S\,3.A]{MR859578} and \cite[\S\,3.13]{MR4577903}, Gromov conjectured that
	\begin{equation}\label{c1.4}
		\parbox{0.8\textwidth}{\centering\textit{if a closed oriented manifold admits a metric of positive scalar curvature, then its simplicial volume vanishes.}}
	\end{equation}
	Guth \cite{MR2753599} and Braun-Sauer \cite{MR4386411} obtained bounds on simplicial volume using a macroscopic version of scalar curvature and volume.  Ma-Yu \cite{ma2025smallscaleindextheory,mayu2026} showed that the simplicial norm of the Poincaré dual of the $\widehat A$-class can be controlled in terms of a scalar curvature lower bound, the volume, and the injectivity radius of the universal cover.

	In this paper, we prove Gromov's conjecture \eqref{c1.4} for all closed oriented manifolds whose universal cover is spin.

	\subsection{Main results}

	Let $M$ be an $m$-dimensional closed oriented manifold. Let $[M]\in H_m(M)$ be its fundamental class and $\widehat{A}(TM)\in H^*(M)$ be its $\widetilde{A}$-class. For $k\in\mathbb{N}$, let $\widehat{A}^{[k]}(TM)\in H^{k}(M)$ denote the $k$-th degree component of $\widehat{A}(TM)$. The cap product $\widehat{A}^{[k]}(TM)\cap [M]\in H_{m-k}(M)$ defines the Poincaré dual of $\widehat{A}^{[k]}(TM)$. In particular, we have $\widehat{A}^{[0]}(TM)\cap [M]=[M]$.
	
	Let $\bV\widehat{A}^{[k]}(TM)\cap[M]\bV_{\ell^1}$ be  \emph{the simplicial norm} of $\widehat{A}^{[k]}(TM)\cap[M]$, defined as the infimum of the sum of the absolute values of the coefficients in a singular cycle representative. In particular, the simplicial norm $\bV[M]\bV_{\ell^1}$ of the fundamental class is called \emph{the simplicial volume}.

	\begin{theo}\label{t1.2'}
		Let $M$ be an $m$-dimensional closed oriented manifold whose universal covering $\widetilde{M}$ is spin. Suppose that M admits a Riemannian metric of positive scalar curvature. Then  the Poincaré dual of each of the component of the $\widehat{A}$-class has vanishing simplicial norm; that is, for any $0\leqslant k\leqslant m$,
		\begin{equation}\label{1.11n}
			\bV\widehat{A}^{[k]}(TM)\cap[M]\bV_{\ell^1}=0.
		\end{equation}
		In particular, when $k=0$, the simplicial volume $\bV[M]\bV_{\ell^1}$ vanishes.
	\end{theo}
	
	Compared with Gromov's conjecture \eqref{c1.4}, Theorem \ref{t1.2'} gives a stronger conclusion for closed manifolds whose universal cover is spin. Rather than asserting only the vanishing of the simplicial volume, it proves the vanishing of the simplicial norm of the Poincaré dual of every component of the $\widehat{A}$-class.

	We also have the following consequence of the main theorem, which, in particular, shows that nonvanishing simplicial volume obstructs the existence of metrics of nonnegative scalar curvature on aspherical manifolds. 
	\begin{coro}\label{t1.2}
		Let $M$ be a closed oriented manifold with spin
		universal cover. If it has positive simplicial volume, then it admits no
		Riemannian metric of nonnegative scalar curvature. In particular, every
		closed oriented aspherical manifold with 
		positive simplicial volume admits no Riemannian metric of nonnegative
		scalar curvature.
	\end{coro}

	% \subsubsection{Simplicial norm vanishing}\label{ISPP}

	% We are now ready to state our main result.

	% We refer to \S\,\ref{LSE} for the proof of Theorem \ref{t1.2'}. 

	\subsection{Techniques and related work}
	
	We now discuss some of the techniques used in the proof and mention additional related work beyond that
	already cited. The topics involved are broad, and the references below are very incomplete.
	
	\subsubsection{Simplicial norms}

	One of the important features of the simplicial norm is its sensitivity to large scale and negatively curved phenomena, and hence its ability to detect geometric and topological complexity. For instance, the simplicial volume vanishes for closed oriented manifolds with an amenable fundamental group, while Gromov-Thurston \cite[\S\,1.2]{MR686042} showed that the simplicial volume of a hyperbolic manifold is proportional to its volume. Gromov \cite{MR636516} used simplicial volume to give a new proof of Mostow rigidity. Since then, the simplicial norm has become an important invariant in geometry, topology, and dynamics. For further developments, see Thurston \cite{MR1435975}, Lück \cite{MR1926649}, Lafont-Schmidt \cite{MR2285319}, Bergeron-Şengün-Venkatesh \cite{MR3513571}, Brock-Dunfield \cite{MR3714511}, Connell-Wang \cite{MR4527829}, and Löh-Moraschini-Raptis \cite{MR4455175}.

	\subsubsection{Scalar curvature geometry and index theory}
	
	The interaction between scalar curvature and topology is a foundational subject in differential geometry, see for example, Schoen-Yau \cite{MR541332,MR1013841}, Gromov-Lawson \cite{MR720933}, Rosenberg \cite{MR720934,MR866507}, Connes \cite{MR866491}, Block-Weinberger \cite{MR1145337,MR1758300}, Stolz \cite{MR1189863}, Roe \cite{MR1147350}, Rosenberg-Stolz \cite{MR1818778} and Zhang \cite{MR3664818,MR4089513}.

	The proof of our main result is inspired by the strategy of Connes-Moscovici \cite{MR1066176} in their approach to the Novikov conjecture. 
	In their seminal work, they proved the Novikov conjecture for groups satisfying property RD, under the additional assumption that every group cohomology class admits a representative cocycle of polynomial growth.
	
	A key ingredient in the present paper is the quantitative index theory developed in Ma-Yu \cite{ma2025smallscaleindextheory,mayu2026}. This theory allows us to use functional calculus of the Dirac operator to construct explicit representatives of the fundamental class and, more generally, of the Poincaré duals of the $\widehat{A}$-class. In this way, it establishes a direct connection between scalar curvature and simplicial norm. Positive scalar curvature allows us to deform these homology cycles, through a continuous deformation of the Dirac operator, to representatives with arbitrarily small simplicial norm. An important new refinement in the proof is the use of heat kernel to derive sharper estimates for the functional calculus.

	Quantitative index theory is closely related to quantitative operator $K$-theory, see Yu \cite{MR1626745,MR1728880,MR3590530}, Oyono-Oyono-Yu \cite{MR3449169}, and Willett-Yu \cite{MR4411373}. These methods have found several applications in the study of topological rigidity, the Novikov conjecture, and scalar curvature geometry, see, for example, Guentner-Tessera-Yu \cite{MR2947546}, Gong-Wu-Yu \cite{MR4268302}, Nowak-Yu \cite{MR4646531}, Wang-Xie-Yu \cite{MR4666628,MR4810224}, and Chen-Wang-Xie-Yu \cite{MR4673948}.

	Guth's surveys \cite{MR2827817,MR3838118} introduce recent results in quantitative topology and their connections with scalar curvature, systolic geometry, and large scale geometry. For further developments, see Guth \cite{MR2753599}, Alpert-Funano \cite{MR3678499}, Chambers-Dotterrer-Manin-Weinberger \cite{MR3836564}, Chambers-Manin-Weinberger \cite{MR3816519}, Braun-Sauer \cite{MR4386411}, and Alpert-Balitskiy-Guth \cite{MR4790653,MR5032981}.

	\subsubsection{A comment on our approach}

	We point out that our argument can also be formulated in the framework of cyclic homology. The present proof, however, works directly with singular homology, making the argument more elementary. We hope that the present paper, together with Ma-Yu \cite{ma2025smallscaleindextheory,mayu2026}, may also serve as an accessible entry for readers interested in noncommutative geometry, Connes-Moscovici's higher index theorem, and index theoretic methods in the study of the Novikov conjecture.

	\subsection{Organization of the article}
	
	In \S\,\ref{SSIT}, we introduce quantitative homology, which interpolates between the homology of $M$ and the group homology of $\Gamma$. In \S\,\ref{LSG}, we recall the quantitative index theorem, which expresses $\widehat A^{[m-k]}(TM)\cap [M]$ by explicit finite propagation operators associated to the Dirac operator. In \S\,\ref{simplicial}, we recall some useful properties of the simplicial norm. In \S\,\ref{LSE}, we combine the quantitative index formula with heat kernel estimates to prove the vanishing of the simplicial norms.

	The symbol $C$ denotes a positive constant. When the dependence of $C$ on certain parameters is important, this will be indicated by a subscript. The value of $C$ may change from line to line, even when the same notation is used.

	\subsection{Acknowledgments}

	We thank Misha Gromov and Nigel Higson for helpful comments and suggestions. Qiaochu Ma is partially supported by NSF DMS-2247313. Jinmin Wang is partially supported by NSFC 12501169. Zhizhang Xie is partially by NSF DMS-2247322. Guoliang Yu is partially supported by NSF DMS-2247313. Bo Zhu is partially supported by NSFC 12501066.

	\section{Quantitative Homology}\label{SSIT}

	In this section, we introduce epsilon homology by Ma-Yu \cite[\S\,2]{ma2025smallscaleindextheory,mayu2026}, a quantitative version of Alexander-Spanier homology. In \S\,\ref{s3.2}, we recall the classifying map and the standard identification of the homology of the classifying space with group homology. In \S\,\ref{S2.1}, we define
	epsilon homology. In \S\,\ref{S2.2}, we show that epsilon homology interpolates between the homology of the manifold and
	the homology of the classifying space.%SASB

	\subsection{Homology and large scale homology}\label{s3.2}

	Let $M$ be a manifold with its universal covering $\widetilde{M}$ and fundamental group $\Gamma=\pi_1(M)$. Let $B\Gamma$ be the classifying space of $\Gamma$ such that $\pi_1(B\Gamma)=\Gamma$ and its universal covering $E\Gamma=\widetilde{B\Gamma}$ is contractible. Then there exists a canonical classifying map $f\colon M\to B\Gamma$ and the following commutative diagram
	\begin{equation}\label{2.1m}
		\begin{tikzcd}[ampersand replacement=\&, column sep=normal,row sep=normal,background color=white!20]
			\widetilde{M}\cong f^*(E\Gamma)\arrow[d,""]\arrow[r,""]\&E\Gamma\arrow[d,""]\\
			M\arrow[r,"f"]\& B\Gamma
		\end{tikzcd},
	\end{equation}
	which induces a morphism at the level of homology
	\begin{equation}\label{2.2m}
		f_*\colon H_*(M)\to H_*(B\Gamma).
	\end{equation}

	The homology $H_*(B\Gamma)$ can be described explicitly in terms of the group homology of $\Gamma$. We now recall the relevant definition.

	Let $C^k(\Gamma)$ be the space of $\Gamma$-invariant functions on $\Gamma^{k+1}$, with $\Gamma$ acting diagonally on $\Gamma^{k+1}$, that is,
	\begin{equation}
		C^k(\Gamma)=\big\{\alpha\mid \alpha(\gamma_0,\cdots\gamma_k)=\alpha(\gamma\gamma_0,\cdots\gamma\gamma_k) \text{ for any }\gamma\in\Gamma\big\}.
	\end{equation}
	Let us define the coboundary map
	\begin{equation}
		\begin{split}
			&\pa\colon C^k(\Gamma)\to C^{k+1}(\Gamma),\\
			&\partial \alpha(\gamma_0,\cdots,\gamma_{k+1})=\sum_{i}(-1)^i\alpha(\gamma_0,\cdots,\widehat{\gamma_i},\cdots\gamma_{k+1}),
		\end{split}
	\end{equation}
	where the ‘hat’ symbol over $\gamma_i$ indicates that this variable is deleted. The group cohomology $H^*(\Gamma)$ of $\Gamma$ is defined by the cohomology of the cochain $(\pa,C^*(\Gamma ))$,
	\begin{equation}
		H^*(\Gamma)=H^*\big(\pa,C^*(\Gamma )\big).
	\end{equation}
	
	Similarly, let $C_k(\Gamma)$ be the space of $\Gamma$-invariant measures on $\Gamma^{k+1}$, with $\Gamma$ acting diagonally on $\Gamma^{k+1}$, that is,
	\begin{equation}
		C_k(\Gamma)=\big\{\mu\mid \mu(\gamma_0,\cdots\gamma_k)=\mu(\gamma\gamma_0,\cdots,\gamma\gamma_k) \text{ for any }\gamma\in\Gamma\big\}.
	\end{equation}
	Let us define the boundary map
	\begin{equation}
		\begin{split}
			&\pa\colon C_k(\Gamma)\to C_{k-1}(\Gamma),\\
			&\partial\delta_{(\gamma_0,\cdots,\gamma_{k})}=\sum_{i}(-1)^i\delta_{(\gamma_0,\cdots,\widehat{\gamma_i},\cdots\gamma_{k})},
		\end{split}
	\end{equation}
	where $\delta_{(\gamma_0,\cdots,\gamma_{k})}$ is the delta measure at $(\gamma_0,\cdots,\gamma_k)\in\Gamma^{k+1}$. The group homology $H_*(\Gamma)$ of $\Gamma$ is defined by the homology of the chain $(\pa,C_*(\Gamma ))$,
	\begin{equation}
		H_*(\Gamma)=H_*\big(\pa,C_*(\Gamma )\big).
	\end{equation}

	The following result is classical, we give the proof since it motivates further discussions.
	\begin{prop}\label{prop2.1m}
		We have
		\begin{equation}\label{3.5}
			H_*(\Gamma)\cong H_*(B\Gamma).
		\end{equation}
	\end{prop}
	
	\begin{pro}
		
		The $\Gamma$-equivariant singular chain complex $(\partial,C_*(E\Gamma,\Gamma))$ of $E\Gamma$ computes the homology $H_*(B\Gamma)$ of $B\Gamma$, namely that
		\begin{equation}
			H_*\bigl(C_*(E\Gamma,\Gamma)\bigr)\cong H_*(B\Gamma).
		\end{equation}
		It suffices to construct a homotopy equivalence between $(\pa,C_*(\Gamma))$ and $(\pa,C_*(E\Gamma,\Gamma))$.
		
		First, since $E\Gamma$ is contractible, we can assign to each $(y_0,\cdots,y_k)\in (E\Gamma)^{k+1}$ a singular simplex $\Delta_{(y_0,\cdots,y_k)}$ having vertices $(y_0,\cdots,y_k)$, such that every subset of $(y_0,\cdots,y_k)$ is assigned to the corresponding face of $\Delta_{(y_0,\ldots,y_k)}$. Moreover, this assignment can be chosen $\Gamma$-equivariantly.

		We choose a bounded fundamental domain $F\subset E\Gamma$ and a base point $y_b\in F$. Define
		\begin{equation}\label{2.11m}
			\begin{split}
				&f^{0,\infty}\colon  C_*(\Gamma)\to C_*(E\Gamma,\Gamma),\\
				&f^{0,\infty}\big(\delta_{(\gamma_0,\cdots,\gamma_k)}\big)=\Delta_{(\gamma_0y_b,\cdots,\gamma_ky_b)}.
			\end{split}
		\end{equation}
		The meaning of the notation $f^{0,\infty}$ will become clear soon in the next subsection.
		
		Conversely, define
		\begin{equation}\label{2.12m}
			\begin{split}
				&f^{\infty,0}\colon C_*(E\Gamma,\Gamma)\to C_*(\Gamma),\\
				&f^{\infty,0}(\Delta)=\delta_{(\gamma_0,\cdots,\gamma_k)},
			\end{split}
		\end{equation}
		where $\Delta$ is a singular simplex with vertices $(y_0,\cdots,y_k)$, and $\gamma_0,\ldots,\gamma_k\in\Gamma$ are chosen so that $(y_0,\ldots,y_k)\in \gamma_0F\times\cdots\times\gamma_kF$. For instance, when $(\Gamma,B\Gamma,E\Gamma,F)=(\mathbb{Z}^m,\mathbb{T}^m,\mathbb{R}^m,[0,1)^m)$, then a simplex $\Delta$ with vertices $(y_0,\cdots,y_k)$ is mapped to $\delta_{(\lfloor y_0\rfloor,\cdots,\lfloor y_k\rfloor)}$, where the floor function is taken coordinatewise.

		The fact that \eqref{2.11m} and \eqref{2.12m} define homotopy equivalent maps follows from the usual prism operator argument in the proof of homotopy invariance of singular homology, see Hatcher \cite[Theorem 2.10]{MR1867354}. In the present setting, the same argument is applied equivariantly. \qed
		
	\end{pro}

	The left hand sides of \eqref{2.1m} and \eqref{2.2m} contain fine small scale homological information, while the right hand sides encode purely large scale homological information. Thus, these suggest the existence of an interpolation describing the transition from small scale to large scale homology, which we introduce in the next subsection.

	\subsection{Epsilon homology}\label{S2.1}

	For any $(x_0,\cdots,x_k)\in \widetilde{M}^{k+1}$, we define its diameter by
	\begin{equation}
		\mathrm{diam}(x_0,\cdots,x_k)=\max_{0\leqslant i,j\leqslant k}d_{\widetilde{M}}(x_i,x_j).
	\end{equation}
	For any $\varepsilon>0$, define the $\varepsilon$-diagonal of $\widetilde{M}^{k+1}$ by
	\begin{equation}\label{diag}
		\mathrm{diag}_\varepsilon(\widetilde{M}^{k+1})=\{(x_0,\cdots,x_k)\in \widetilde{M}^{k+1}\mid \mathrm{diam}(x_0,\cdots,x_k)\leqslant\varepsilon\}.
	\end{equation}

	Let $C^{k}_\varepsilon(\widetilde{M},\Gamma)$ be the space of $\Gamma$-invariant Borel functions on $\mathrm{diag}_\varepsilon(\widetilde{M}^{k+1})$. We define a coboundary map
	\begin{equation}\label{2.4}
		\begin{split}
			&\pa\colon C^{k-1}_\varepsilon(\widetilde{M},\Gamma)\to C^{k}_\varepsilon(\widetilde{M},\Gamma),\\
			&(\pa f)(x_0,\cdots,x_k)=\sum_{i=0}^{k}(-1)^if(x_0,\cdots,\widehat{x}_i,\cdots,x_k),
		\end{split}
	\end{equation}
	where the ‘hat’ symbol over $x_i$ indicates that this variable is deleted.

	The epsilon cohomology  $H^*_\varepsilon(\widetilde{M},\Gamma)$ is defined as the cohomology of the cochain complex $(\pa,C^{*}_\varepsilon(\widetilde{M},\Gamma))$, that is,
	\begin{equation}
		H^*_\varepsilon(\widetilde{M},\Gamma)=H(\pa,C^{*}_\varepsilon(\widetilde{M},\Gamma)).
	\end{equation}

	Likewise, let $C_{k}^\varepsilon(\widetilde{M},\Gamma)$ be the space of $\Gamma$-invariant measures supported on $\mathrm{diag}_\varepsilon(\widetilde{M}^{k+1})\subseteq \widetilde{M}^{k+1}$.

	By viewing $\alpha\in C^*_\varepsilon(\widetilde{M},\Gamma)$ as a continuous function on $\Gamma\backslash \widetilde{M}^{*+1}$ and $\mu\in C_*^\varepsilon(\widetilde{M},\Gamma)$ as a measure on $\Gamma\backslash \widetilde{M}^{*+1}$, we can define a pairing $\langle\cdot,\cdot\rangle_\Gamma$ between $C^*_\varepsilon(\widetilde{M},\Gamma)$ and $C_*^\varepsilon(\widetilde{M},\Gamma)$ by the integral of $\alpha$ with respect to $\mu$ over $\Gamma\backslash \widetilde{M}^{*+1}$. In other words, let $F\subset \widetilde{M}$ be a bounded fundamental domain of $\Gamma$-action, then $F\times \widetilde{M}^k$ is a fundamental domain of the diagonal $\Gamma$-action on $\widetilde{M}^{k+1}$, and
	\begin{equation}\label{2.10''}
		\langle\alpha,\mu\rangle_\Gamma=\int_{F\times \widetilde{M}^k}\alpha d\mu.
	\end{equation}
	Note that $\mathrm{diag}_\varepsilon(\widetilde{M}^{k+1})\cap \big(F\times \widetilde{M}^k\big)$ is compact, so the integral \eqref{2.10''} is finite.

	We define the dual boundary map
	\begin{equation}
		\begin{split}
			&\pa\colon C_{k}^\varepsilon(\widetilde{M},\Gamma)\to C_{k-1}^\varepsilon(\widetilde{M},\Gamma),\\
			&\langle \pa\mu,\alpha\rangle_\Gamma=\langle \mu,\pa \alpha\rangle_\Gamma,
		\end{split}
	\end{equation}
	where $\partial \alpha$ is given in \eqref{2.4}. We obtain a chain complex $(\pa,C_{*}^\varepsilon(\widetilde{M},\Gamma))$. We define the epsilon homology
	$H_*^\varepsilon(\widetilde{M},\Gamma)$ to be the homology of the chain complex $(\pa,C_{*}^\varepsilon(\widetilde{M},\Gamma))$, that is,
	\begin{equation}
		H_*^\varepsilon(\widetilde{M},\Gamma)=H(\pa,C_{*}^\varepsilon(\widetilde{M},\Gamma)).
	\end{equation}

	For $\varepsilon\leqslant \varepsilon'$, we have an inclusion from $C^{\varepsilon}(\widetilde{M},\Gamma)$ to $C^{\varepsilon'}(\widetilde{M},\Gamma)$, and we denote the induced morphism at the level of homology by
	\begin{equation}
		f^{\varepsilon',\varepsilon}\colon H^{\varepsilon}_*(\widetilde{M},\Gamma)\to H^{\varepsilon'}_*(\widetilde{M},\Gamma).
	\end{equation}
	Then for $\varepsilon\leqslant \varepsilon'\leqslant\varepsilon''$, we have
	\begin{equation}
		f^{(\varepsilon'',\varepsilon')}f^{(\varepsilon',\varepsilon)}=f^{(\varepsilon'',\varepsilon)}.
	\end{equation}
	In other words, we have the following commutative diagram
	\begin{equation}\label{2.22m}
		\begin{tikzcd}[ampersand replacement=\&, column sep=normal,row sep=normal,background color=white!20]
			H^{\varepsilon}_*(\widetilde{M},\Gamma)\arrow[rr,"f^{(\varepsilon'',\varepsilon)}"swap,bend right=20]\arrow[r,"f^{(\varepsilon',\varepsilon)}"]\& H^{\varepsilon'}_*(\widetilde{M},\Gamma)\arrow[r,"f^{(\varepsilon'',\varepsilon')}"] \& H^{\varepsilon''}_*(\widetilde{M},\Gamma)
		\end{tikzcd}.
	\end{equation}

	\subsection{Epsilon homology as an interpolation}\label{S2.2}

	Intuitively, epsilon homology $H^\varepsilon(\widetilde{M},\Gamma)$ provides an interpolation between $H_*(M)$ and $H_*(B\Gamma)$ as $\varepsilon$ changes from $0$ to $\infty$. To see this, we first extend \eqref{2.22m} to limits.

	The $\Gamma$-equivariant singular chain complex $(\partial,C_*(\widetilde M,\Gamma))$ of $\widetilde{M}$ computes the homology $H_*(M)$ of $M$, namely that
	\begin{equation}
		H_*\bigl(C_*(\widetilde M,\Gamma)\bigr)\cong H_*(M).
	\end{equation}
	We define
	\begin{equation}\label{2.12}
		\begin{split}
			&f^{\varepsilon,0}\colon C_*(\widetilde{M},\Gamma)\to C_*^\varepsilon (\widetilde{M},\Gamma),\\
			&f^{\varepsilon,0}(\Delta)=\delta_{(x_0,\cdots,x_k)},
		\end{split}
	\end{equation}
	where $\Delta$ is a simplex with vertices $(x_0,\cdots,x_k)\in \widetilde{M}^{k+1}$, and $\delta_{(x_0,\cdots,x_k)}$ is the delta measure at $(x_0,\cdots,x_k)$.  Note that every singular simplex can be subdivided into simplices whose vertices have diameter no more than $\varepsilon$, so the map $f^{\varepsilon,0}$ is well defined at the level of homology.

	We choose a bounded fundamental domain $F\subset \widetilde{M}$ and define
	\begin{equation}\label{2.13}
		\begin{split}
			&f^{\infty,\varepsilon}\colon C_*^\varepsilon (\widetilde{M},\Gamma)\to C_*(\Gamma),\\ &f^{\infty,\varepsilon}(\mu)=\sum_{\gamma_0,\cdots,\gamma_k\in\Gamma}\big\langle\mu,\mathbbm{1}_{(\gamma_0F)\times\cdots\times(\gamma_kF)}\big\rangle_\Gamma\delta_{(\gamma_0,\cdots,\gamma_k)},
		\end{split}
	\end{equation}
	where
	\begin{equation}
		\mathbbm{1}_{(\gamma_0F)\times\cdots\times(\gamma_kF)}(x_0,\cdots,x_k)=\begin{cases}
			1,&\text{if }x_i\in \gamma_iF \text{ for every }0\leqslant i\leqslant k,\\
			0,&\text{otherwise}.
		\end{cases}
	\end{equation}

	Then we define the homological radius $r(M)$ of $M$ by the supremum of all $\varepsilon>0$ with the following property, for any $x\in\widetilde{M}$, the radius $\varepsilon$ ball $B^{\widetilde{M}}(x,\varepsilon)$ is contractible.
	
	We now explain precisely that $H^\varepsilon(\widetilde{M},\Gamma)$ provides the desired interpolation between $H_*(M)$ and $H_*(B\Gamma)$.
	
	\begin{prop}\label{prop2.1}
		Let $M$ be a closed manifold with universal covering $\widetilde{M}$, fundamental group $\Gamma$, and canonical classifying map $f\colon M\to B\Gamma$. For $0<\varepsilon< r(M)$, the map $f^{\varepsilon,0}$ defined in \eqref{2.12} is a homotopy equivalence, in particular,
		\begin{equation}
			H^\varepsilon(\widetilde{M},\Gamma)\cong H_*(M).
		\end{equation}
		Moreover, for any $\varepsilon>0$ and $f^{\infty,\varepsilon}$ defined in \eqref{2.13},
		\begin{equation}\label{2.28m}
			f^{\infty,\varepsilon}f^{\varepsilon,0}=f_*.
		\end{equation}
		Equivalently, we have the following commutative diagram
		\begin{equation}
			\begin{tikzcd}[ampersand replacement=\&, column sep=normal,row sep=large,background color=white!20]
				\varprojlim_{\varepsilon\to0}H^{\varepsilon}_*(\widetilde{M},\Gamma)\arrow[r]\& H_*^\varepsilon(\widetilde{M},\Gamma)\arrow[dashed,l,"\varepsilon<r(M)"swap,bend right=30]\arrow[r]\arrow[rd,"f^{(\infty,\varepsilon)}"swap]\&\varinjlim_{\varepsilon\to\infty}H^{\varepsilon}_*(\widetilde{M},\Gamma)\arrow[d,"\cong"] \\
				H_*(M)\arrow[ru,"f^{(\varepsilon,0)}"swap]\arrow[u,"\cong"]\arrow[rr,"f_*"swap]\&  \&H_*(B\Gamma)\cong H_*(\Gamma)
			\end{tikzcd},
		\end{equation}
		where the dashed arrow denotes a homotopy inverse when $\varepsilon<r(M)$.
	\end{prop}
	
	\begin{pro}
		To construct a homotopy inverse, we use a local analogue of the argument in Proposition \ref{prop2.1m}. In the aspherical case, the key point was that the contractibility of $E\Gamma$ allowed us to reconstruct simplices from their vertices in a $\Gamma$-equivariant way. Here we only need such a construction for the $\varepsilon$-diagonal $\mathrm{diag}_\varepsilon(\widetilde{M}^{k+1})$ defined in \eqref{diag}. Indeed, when $0<\varepsilon<r(M)$, the definition of the homological radius ensures that the inclusion
		\begin{equation}
			\mathrm{diag}(\widetilde{M}^{k+1})=\mathrm{diag}_0(\widetilde{M}^{k+1})\to\mathrm{diag}_\varepsilon(\widetilde{M}^{k+1}),
		\end{equation}
		is homotopic, in a way compatible with faces and $\Gamma$-equivariant. Then every $(x_0,\cdots,x_k)\in\widetilde{M}^{k+1}$ with $\mathrm{diam}(x_0,\cdots,x_k)<r(M)$ can be filled by a simplex $\Delta_{(x_0,\cdots,x_k)}$ compatibly with faces and with the $\Gamma$-action. This provides the homotopy inverse to $f^{\varepsilon,0}$.
		
		We can check \eqref{2.28m} directly from \eqref{2.1m}, \eqref{3.5}, \eqref{2.12}, and \eqref{2.13}.\qed
	\end{pro}

	\section{Quantitative Index Theorem}\label{LSG}

	In this section, we recall the quantitative index theorem by Ma-Yu \cite{ma2025smallscaleindextheory,mayu2026}. In \S\,\ref{S2.3}, we introduce a weakened spin condition, following Rosenberg \cite{MR720934}. In \S\,\ref{S2.4}, we construct a finite propagation representative of the higher index following Roe \cite{MR996446}. In \S\,\ref{S3.3m}, we state the quantitative index theorem, which gives an explicit expression for the class  $\widehat A(TM)\cap[M]$ in $H_*^\varepsilon(\widetilde{M},\Gamma)$.
	
	Throughout this paper, we  focus our attention on the even dimensional case for the quantitative index theorem. The odd dimensional quantitative higher index theorem follows by a standard suspension argument. 
	
	\subsection{Weakened spin condition}\label{S2.3}

	We now suppose that $\widetilde{M}$ is spin, while $M$ itself is not necessarily spin.  We use a construction of Rosenberg \cite[\S\,3B]{MR720934} to handle this weaker spin condition.
	
	Let $P_{\widetilde{M}}^{\mathrm{SO}_m}$ be the principal bundle of orthonormal oriented frames on $\widetilde{M}$ and $P_{\widetilde{M}}^{\mathrm{Spin}_m}$ the principal spin bundle. By the path lifting property, we can define a double covering $\widetilde{\Gamma}$ of $\Gamma$, using the pullback diagram
	\begin{equation}\label{fig1}
		\begin{tikzcd}[ampersand replacement=\&, column sep=normal,row sep=normal,background color=white!20]
			\widetilde{\Gamma}\arrow[r,""swap]\arrow[d,""swap]\& \mathrm{Aut}(P_{\widetilde{M}}^{\mathrm{Spin}_m})\arrow[d,""swap]\\
			\Gamma \arrow[r,""swap]\& \mathrm{Aut}(P_{\widetilde{M}}^{\mathrm{SO}_m})
		\end{tikzcd},
	\end{equation}
	then $\widetilde{\Gamma}$ acts by unitary bundle automorphisms on the spinor bundle $S_{\widetilde{M}}$ as well as $L^2\big(\widetilde{M},S_{\widetilde{M}}\big)$.
	
	Let $\mathscr{K}^{\widetilde{\Gamma}}$ be the space of $\widetilde{\Gamma}$-invariant smooth kernel operators on $L^2\big(\widetilde{M},S_{\widetilde{M}}\big)$, whose elements are of the form $K=K(x,y)\in S_{\widetilde{M},x}\otimes S_{\widetilde{M},y}^*$ with
	\begin{equation}\label{2.16}
		K(x,y)=\widetilde{\gamma}^{-1}K\big({\gamma}^{} x,{\gamma}^{} y\big)\widetilde{\gamma}^{},
	\end{equation} 
	where $\widetilde{\gamma}\in \widetilde{\Gamma}$ and $\gamma\in\Gamma$ is the image of $\widetilde{\gamma}$.
	
	Although $K\in\mathscr{K}^{\widetilde{\Gamma}}$ is not $\Gamma$-invariant, its pointwise trace is $\Gamma$-invariant in the sense that for any $\gamma\in\Gamma$,
	\begin{equation}
		\mathrm{Tr}^{S_{\widetilde{M}}}\big[K(x,x)\big]=\mathrm{Tr}^{S_{\widetilde{M}}}\big[K(\gamma x,\gamma x)\big],
	\end{equation} 
	which is a direct consequence of the fact that $\widetilde{\Gamma}$ is a double covering of $\Gamma$.

	Let $\DM$ be the Dirac operator on $\widetilde{M}$, then $\DM$ is $\widetilde{\Gamma}$-invariant. Moreover, for any $\phi\in\mathscr{S}(\mathbb{R})$, the Schwartz space, the functional calculus $\phi(\DM)$ is a smooth kernel operator and $\phi(\DM)\in\mathscr{K}^{\widetilde{\Gamma}}$.

	Moreover, we recall the Schrödinger-Lichnerowicz formula \cite{MR4045483,MR156292}, which plays a fundamental role in the index  theoretic approach to scalar curvature geometry,
	\begin{equation}\label{2.18n}
		\DM^2=\Delta^{S_{\widetilde{M}}}+\frac{1}{4}\mathrm{Sc}_{g^{T\widetilde{M}}},
	\end{equation}
	where $\Delta^{S_{\widetilde M}}$ is the Bochner Laplacian, and $\mathrm{Sc}_{g^{T\widetilde{M}}}$ is the scalar curvature.
	
	\subsection{Finite propagation index representative}\label{S2.4}

	We now follow Roe's construction \cite[Lemma 7.5]{MR996446}. Choose a smooth function $\phi_{\varepsilon,0}(\lambda)$ such that
	\begin{equation}\label{2.18}
		\begin{split}
			&\phi_{\varepsilon,0}(\lambda) \text{ odd},\quad\phi_{\varepsilon,0}'(\lambda) \text{ even},\quad \phi_{\varepsilon,0}'(\lambda)\in\mathscr{S}(\mathbb{R}),\\
			&\phi_{\varepsilon,0}(\pm\infty)=\pm1,\quad \mathrm{supp}(\widehat{\phi_{\varepsilon,0}'}(\xi))\subseteq [-\tfrac{\varepsilon}{6m},\tfrac{\varepsilon}{6m}],
		\end{split}
	\end{equation} 
	where $\widehat{\phi_{\varepsilon,0}'}$ 
	is the Fourier transform of $\phi_{\varepsilon,0}'$. Then we have a smooth $\phi_{\varepsilon,1}(\lambda)$ such that
	\begin{equation}\label{2.19}
		\begin{split}
			&\phi_{\varepsilon,1}(\lambda) \text{ even},\quad \phi_{\varepsilon,1}(\lambda)\in\mathscr{S}(\mathbb{R}),\\
			&\phi_{\varepsilon,1}(\lambda)=\phi_{\varepsilon,0}(\lambda)^2-1,\quad\mathrm{supp}(\widehat{\phi_{\varepsilon,1}}(\xi))\subseteq [-\tfrac{\varepsilon}{3m},\tfrac{\varepsilon}{3m}].
		\end{split}
	\end{equation}
	Indeed, in the sense of distributions, we have
	\begin{equation}
		\sqrt{-1}\xi\widehat{\phi_{\varepsilon,0}}(\xi)=\widehat{\phi_{\varepsilon,0}'}(\xi),\quad \xi\widehat{\phi_{\varepsilon,1}}(\xi)=(\widehat{\phi_{\varepsilon,0}}*\widehat{\phi_{\varepsilon,0}'})(\xi),
	\end{equation}
	where the convolution is well defined since the distributions have compact support, see Rudin \cite[\S\S\,6, 7]{MR1157815} for details.

	Let $I_\varepsilon$ be the ``difference idempotent" defined by
	\begin{equation}\label{2.21}
		\begin{split}
			I_\varepsilon=\begin{pmatrix} \phi_{\varepsilon,1}^+(\DM)^2&-\phi_{\varepsilon,1}^+( \DM)\big(1-\phi_{\varepsilon,1}^+(\DM)\big)\phi_{\varepsilon,0}^-(\DM)\\
				-\phi_{\varepsilon,1}^-(\DM)\phi_{\varepsilon,0}^+(\DM)&-\phi_{\varepsilon,1}^-(\DM)^2
			\end{pmatrix},
		\end{split}
	\end{equation}
	where the superscripts $\pm$ denote the restriction to $L^2(\widetilde{M},S_{\widetilde{M}}^\pm)$. Note that in operator $K$-theory, $I_\varepsilon$ can be regarded as a small scale representative of the higher index of the Dirac operator $\DM$.
	
	Combining \eqref{2.18}, \eqref{2.19}, and \eqref{2.21} with the finite propagation speed of the wave operator, we conclude that $I_\varepsilon$ has propagation at most
	$\tfrac{\varepsilon}{m}$, that is, for all  $x\in \widetilde{M}$,
	\begin{equation}\label{2.22}
		\supp (I_\varepsilon(x,\cdot))\subseteq B^{\widetilde{M}}(x,\tfrac{\varepsilon}{m}),
	\end{equation}
	where $I_\varepsilon(x,y)$ is the kernel of $I_\varepsilon$.

	\subsection{Quantitative index theorem}\label{S3.3m}

	Since $I_\varepsilon\in \mathscr{K}^{\widetilde{\Gamma}}$, we have for any $\widetilde{\gamma}\in\widetilde{\Gamma}$, the cyclic product index kernel $I_\varepsilon\big(x_{0},x_{1}\big)\cdots I_\varepsilon\big(x_{k-1},x_{k}\big)I_\varepsilon\big(x_{k},x_{0}\big)$ satisfies
	\begin{equation}\label{2.23}
		\begin{split}
			&I_\varepsilon\big(x_{0},x_{1}\big)\cdots I_\varepsilon\big(x_{k-1},x_{k}\big)I_\varepsilon\big(x_{k},x_{0}\big)\\
			&=\widetilde{\gamma}I_\varepsilon\big(\gamma x_{0},\gamma x_{1}\big)\cdots I_\varepsilon\big(\gamma x_{k-1},\gamma x_{k}\big)I_\varepsilon\big(\gamma x_{k},\gamma x_{0}\big)\widetilde{\gamma}^{-1}.
		\end{split}
	\end{equation}
	Hence its pointwise trace is $\Gamma$-invariant, that is,
	\begin{equation}
		\begin{split}
			&\mathrm{Tr}^{S_{\widetilde{M}}}\big[I_\varepsilon\big(x_{0},x_{1}\big)\cdots I_\varepsilon\big(x_{k-1},x_{k}\big)I_\varepsilon\big(x_{k},x_{0}\big)\big]\\
			&=\mathrm{Tr}^{S_{\widetilde{M}}}\big[I_\varepsilon\big(\gamma x_{0},\gamma x_{1}\big)\cdots I_\varepsilon\big(\gamma x_{k-1},\gamma x_{k}\big)I_\varepsilon\big(\gamma x_{k},\gamma x_{0}\big)\big].
		\end{split}
	\end{equation}
	Moreover, by \eqref{2.22}, we have
	\begin{equation}
		\mathrm{supp}\big(I_\varepsilon\big(x_{0},x_{1}\big)\cdots I_\varepsilon\big(x_{k-1},x_{k}\big)I_\varepsilon\big(x_{k},x_{0}\big)\big)\subseteq \mathrm{diag}_\varepsilon(\widetilde{M}^{k+1}).
	\end{equation}
	
	%We note that $L_\varepsilon,P_\varepsilon,I_\varepsilon\in \mathscr{K}^{\widetilde{\Gamma}}$.

	We now define the Connes-Chern chain $\mathrm{Ch}_{k}^\varepsilon\big(\DM\big)\in H_{k}^\varepsilon(\widetilde{M},\Gamma)$ of $\DM$ as an antisymmetric measure
	\begin{equation}\label{3.13}
		\begin{split}
			&\mathrm{Ch}_{k}^\varepsilon\big(\DM\big)\qquad\qquad\qquad\qquad\qquad\qquad\qquad\qquad\qquad\qquad\quad\\
			&=\sum_{\sigma\in \mathfrak{S}_{k+1}}\mathrm{sgn}(\sigma)\mathrm{Tr}^{S_{\widetilde{M}}}\big[I_\varepsilon\big(x_{\sigma(0)},x_{\sigma(1)}\big)\cdots I_\varepsilon\big(x_{\sigma(k)},x_{\sigma(0)}\big)\big]dv_{\widetilde{M}^{k+1}}(x_{0},\cdots,x_{k}),
		\end{split}
	\end{equation}
	where $\mathfrak{S}_{k+1}$ denotes the permutation group on $(k+1)$ numbers, and $\mathrm{sgn}(\sigma)$ denotes the sign of $\sigma$.

	Before stating the index theorem, we recall the characteristic class $\widehat{A}(TM)\in H^*(M)$. For $k\in\mathbb{N}$, let $\widehat{A}^{[k]}(TM)\in H^{k}(M)$ denote the $k$-th degree component of $\widehat{A}(TM)$. Let $R^{TM}$ denote the  Riemannian curvature associated to the Riemannian metric $g^{TM}$ on $M$. Then $\widehat{A}(TM)$ can be represented by
	\begin{equation}\label{eq:ahat}
		\begin{split}
			\widehat{A}(TM)
			&=\det{}^{1/2}\Bigg(\frac{\tfrac{1}{2\pi i}R^{TM}/2}{\sinh\Big(\tfrac{1}{2\pi i}R^{TM}/2\Big)}\Bigg)\\
			&=\widehat{A}^{[0]}(TM)+\widehat{A}^{[4]}(TM)+\widehat{A}^{[8]}(TM)+\cdots\\
			&=1-\frac{1}{24}p_1(TM)+\Big(\frac{7}{5760}p_1(TM)^2-\frac{1}{1440}p_2(TM)\Big)+\cdots
		\end{split}
	\end{equation}
	where $p_k(TM)\in H^{4k}(M)$ is the $k$-th Pontryagin class of $M$.

	The cap product $\widehat{A}^{[k]}(TM)\cap [M]\in H_{m-k}(M)$ defines the Poincaré dual of $\widehat{A}^{[k]}(TM)$. Note that in particular, we have
	\begin{equation}\label{1.5x}
		\widehat{A}^{[0]}(TM)\cap [M]=[M],\quad\widehat{A}^{[m]}(TM)\cap [M]=\widehat{A}(M),
	\end{equation}
	which are respectively the fundamental class and the $\widehat A$-genus.

	We now state the quantitative index theorem of Ma-Yu \cite{ma2025smallscaleindextheory,mayu2026}.
	\begin{theo}\label{T3.2}
		Let $M$ be a closed manifold whose universal cover is spin. Let $\Gamma$ denote its fundamental group and $f\colon M\to B\Gamma$ the canonical classifying map. For any $\varepsilon>0$ and  $f^{\varepsilon,0}$ defined in \eqref{2.12}, we have
		\begin{equation}\label{3.15..}
			f^{\varepsilon,0}\big(\widehat{A}^{[m-k]}(T{M})\cap[M]\big)=\mathrm{Ch}_{k}^\varepsilon\big(\DM\big).
		\end{equation}
		Moreover, for $f^{\infty,\varepsilon}$ defined in \eqref{2.13}, in view of \eqref{2.28m}, we have
		\begin{equation}\label{3.15m}
			\begin{split}
				f_*\big(\widehat{A}^{[m-k]}(T{M})\cap[M]\big)&=f^{\infty,\varepsilon}f^{\varepsilon,0}\big(\widehat{A}^{[m-k]}(T{M})\cap[M]\big)\\
				&=f^{\infty,\varepsilon}\mathrm{Ch}_{k}^\varepsilon\big(\DM\big).
			\end{split}
		\end{equation}
		Equivalently, we have the following commutative diagram
		\begin{equation}
			\begin{tikzcd}[ampersand replacement=\&, column sep=normal,row sep=large,background color=white!20]
				{\begin{gathered}
						\widehat{A}^{[m-k]}(TM)\cap[M]\\ 
						\in H_*(M)
					\end{gathered}\arrow[r,"f^{(\varepsilon,0)}"]}\arrow[rr,"f_*"swap,bend right=25]\&{\begin{gathered}f^{(\varepsilon,0)}(\widehat{A}^{[m-k]}(TM)\cap[M])\\
						=\mathrm{Ch}_{k}^\varepsilon\big(\DM\big)\\
						\in H_*^\varepsilon(\widetilde{M},\Gamma)
				\end{gathered}} \arrow[r,"f^{(\infty,\varepsilon)}"] \&{\begin{gathered}
						f_*(\widehat{A}^{[m-k]}(TM)\cap[M])\\
						=f^{(\infty,\varepsilon)}\mathrm{Ch}_{k}^\varepsilon\big(\DM\big)\\ \in H_*(B\Gamma)=H_*(\Gamma)
				\end{gathered}}
			\end{tikzcd}.
		\end{equation}
	\end{theo}

	Theorem \ref{T3.2} can be viewed as a simultaneous quantitative refinement of Connes-Moscovici's higher index theorem and $\Gamma$-higher index theorem \cite[Theorems 3.9, 5.4]{MR1066176}. Indeed, identity \eqref{3.15..} corresponds to the higher index theorem, while identity \eqref{3.15m} corresponds to the higher $\Gamma$-index theorem. Identity \eqref{3.15m} will play an important role in our estimate of simplicial norms, so we now spell it out explicitly.
	
	Let $I_\varepsilon$ be as in \eqref{2.21}, and let $F\subset \widetilde{M}$ be a bounded fundamental domain. For $\varepsilon>0$, $\gamma,\gamma'\in\Gamma$, we define the localized component $I_{\varepsilon,\gamma,\gamma'}$ of $I_\varepsilon$ by
	\begin{equation}\label{3.8}
		\begin{split}
			&I_{\varepsilon,\gamma,\gamma'}\colon L^2(\gamma'F,S_{\widetilde{M}})\to L^2(\gamma F,S_{\widetilde{M}}),\\
			&I_{\varepsilon,\gamma,\gamma'}s(x)=\int_{\gamma'F}I_{\varepsilon}(x,x')s(x')dv_{\widetilde{M}}(x').
		\end{split}
	\end{equation}
	By the equivariance \eqref{2.16} of $I_\varepsilon$, we have for any $\widetilde{\gamma}\in\widetilde{\Gamma}$,
	\begin{equation}\label{3.9}
		I_{\varepsilon,\gamma''\gamma,\gamma''\gamma'}(\gamma''x,\gamma''x')=\widetilde{\gamma}''I_{\varepsilon,\gamma,\gamma'}(x,x')(\widetilde{\gamma}'')^{-1}.
	\end{equation}
	
	Combining \eqref{2.13}, \eqref{3.13}, and \eqref{3.15..}, we obtain the following explicit form of the quantitative higher $\Gamma$-index theorem.
	
	\begin{theo}\label{gammaindex}
		Let $M$ be a closed manifold with a spin universal covering $\widetilde{M}$. Let $\Gamma$ denote its fundamental group and $f\colon M\to B\Gamma$ the canonical classifying map. Then, for any $\varepsilon>0$,
		\begin{equation}\label{3.10}
			\begin{split}
				&f_*\big(\widehat{A}^{[m-k]}(T{M})\cap[M]\big)\\
				&=\sum_{\substack{\gamma_0=\gamma_{k+1}=1_\Gamma,\\ \gamma_1,\cdots,\gamma_{k}\in\Gamma,\\ \sigma\in \mathfrak{S}_{k+1}}}\mathrm{sgn}(\sigma)\mathrm{Tr}^{L^2(\gamma_{\sigma(0)}F,S_{\widetilde{M}})}\big(I_{\varepsilon,\gamma_{\sigma(k+1)},\gamma_{\sigma(k)}}\cdots I_{\varepsilon,\gamma_{\sigma(1)},\gamma_{\sigma(0)}}\big)\delta_{(\gamma_{0},\cdots,\gamma_{k})},
			\end{split}
		\end{equation}
		where $\mathfrak{S}_{k+1}$ denotes the permutation group on $(k+1)$ numbers, $\mathrm{sgn}(\sigma)$ denotes the sign of $\sigma$, and $\sigma(k+1)=\sigma(0)$. In particular, when $k=m$, we have $f_*\big(\widehat{A}^{[0]}(T{M})\cap[M]\big)=f_*[M]$.
	\end{theo}

	We note that \eqref{3.10} uses the standard formula expressing the $L^2$-trace of a kernel operator as the integral of the pointwise trace of its kernel. The fact that the fundamental domain $F\subset\widetilde M$ is not necessarily a smooth region causes no difficulty. Indeed, after choosing a partition of unity on $M$ and lifting it to $\widetilde M$, the computation reduces locally to the usual trace formula for smooth compactly supported kernels.

	%In this section, we discuss the large scale geometry of $\widetilde M$. The key observation is that, for a bounded fundamental domain $F\subset\widetilde{M}$, we have $\widetilde M=\Gamma F$. Thus, at large scale, we have $\widetilde{M}\approx\Gamma$. We explain how the metric, homology, operators, and index theorem on $\widetilde M$ can be transferred correspondingly to $\Gamma$. In \S\,\ref{s3.1}, we recall the word length on $\Gamma$ and the reduced group $C^*$-algebra of $\Gamma$. In \S\,\ref{s3.2}, we discuss the (co)homology of $\Gamma$. In \S\,\ref{s3.4}, we discuss the large scale decomposition of the index theorem. 

	\section{Simplicial Norm}\label{simplicial}

	In \S\,\ref{exam}, we recall the definition of simplicial norm and give some examples. In \S\,\ref{mapping}, we recall Gromov's mapping theorem. In \S\,\ref{s5.0}, we get useful estimates for simplicial norms.

	\subsection{Definition and Examples}\label{exam}

	In \cite{MR686042}, Gromov introduced the \emph{simplicial norm}, together with its important special case, \emph{the simplicial volume}. We now recall the definition. 
	
	For a singular homology class $\omega\in H_*(M)$ (with coefficients in $\mathbb{R}$ or $\mathbb{C}$), its simplicial norm $\lV\omega \rV_{\ell^1}$ is defined by
	\begin{equation}\label{gromovnorm}
		\lV\omega\rV_{\ell^1}=\inf_{[\sum_{i}a_i\Delta_i]=\omega} \sum_i \lv a_i\rv,
	\end{equation}
	where the infimum is taken over all singular cycles $\sum_{i}a_i\Delta_i$ representing $\omega$. Equivalently, $\lV\omega\rV_{\ell^1}$ is the infimal number of simplices that we need to represent $\omega$. In particular, the simplicial norm $\bV[M]\bV_{\ell^1}$ of the fundamental class $[M]\in H_m(M)$ is called the simplicial volume of $M$.

	%Recall the definition of the simplicial norm from \eqref{gromovnorm}. 
	
	We discuss two elementary examples that explain the behavior of this invariant.	
	
	First, the simplicial volume $\lV[S^1]\rV_{\ell^1}$ of the circle $S^1$ vanishes. Indeed, let $\Delta_{100}$ be a singular $1$-simplex winding $100$ times around $S^1$, then $[S^1]=\frac{1}{100}\Delta_{100}$, and by \eqref{gromovnorm}, we have $\lV[S^1]\rV_{\ell^1}\leqslant \frac{1}{100}$. Replacing $100$ by an arbitrarily large number, we see that $\lV[S^1]\rV_{\ell^1}=0$. More generally, Gromov \cite[\S\,3.0]{MR686042} showed that if $M$ has amenable fundamental group, then
	\begin{equation}
		\bV[M]\bV_{\ell^1}=0.
	\end{equation}

	In contrast, Gromov-Thurston \cite[\S\,1.2]{MR686042} proved that if $M$ is a hyperbolic manifold, then there is a constant $C_m>0$ such that
	\begin{equation}
		\lV[M]\rV_{\ell^1}=C_m\mathrm{Vol}(M).
	\end{equation}
	This formula shows that the volume of a hyperbolic manifold is a homotopy invariant, a fact lying at the heart of Mostow rigidity. To gain a sense of intuition, we briefly recall the geometric idea in dimension two. A priori, a singular simplex appearing in a cycle may have a very complicated image and boundary. However, in a hyperbolic surface, we can replace each simplex by the geodesic simplex with the same vertices, and this does not change the coefficients. Since this cycle still represents the fundamental class, its total area must cover the area of the surface, and hence 
	\begin{equation}
		\sum_i \lv a_i\rv\mathrm{Area}(\Delta_i)\geqslant \mathrm{Area}(M).
	\end{equation}
	A classical result in hyperbolic geometry, dating back to the pioneering work of Gauss, Lobachevsky, and Bolyai nearly two centuries ago, says that every geodesic triangle in the hyperbolic plane has area bounded above (by $\pi$), in sharp contrast with the Euclidean case. We then obtain
	\begin{equation}
		\sum_i\lv a_i\rv\geqslant \mathrm{Area}(M)/\pi.
	\end{equation}
	This gives the basic reason why the simplicial volume of a hyperbolic surface is proportional to its area.

	These examples show that simplicial volume is sensitive to large scale phenomena.

	\subsection{Gromov's mapping theorem}\label{mapping}

	The following result will be useful when applying Theorem \ref{gammaindex} to estimates of simplicial norms.
	
	\begin{prop}
		Let $M$ be a closed manifold with universal covering $\widetilde{M}$, fundamental group $\Gamma$, and canonical classifying map $f\colon M\to B\Gamma$. Then for any homology class $\omega\in H_*(M)$, we have
		\begin{equation}\label{4.1m}
			\lV\omega\rV_{\ell^1}=\lV f_*\omega\rV_{\ell^1}.
		\end{equation}
	\end{prop}
	
	\begin{pro}
		Gromov's mapping theorem \cite[\S\,3.1]{MR686042} shows that the classifying map $f$ induces an isometric isomorphism on bounded cohomology. On the other hand, by Gromov's duality principle \cite[\S\,1.1]{MR686042}, the simplicial norm can also be defined by pairing with bounded cohomology. Hence, $f$ also preserves simplicial norms.\qed
	\end{pro}

	\subsection{Simplicial norm estimates}\label{s5.0}

	The following estimate for the simplicial norm plays a key role in the proof of the main theorem.

	\begin{prop}\label{inter}
		Let $(M,g^{TM})$ be an $m$-dimensional closed oriented Riemannian manifold whose universal covering $\widetilde{M}$ is spin. Then for any $k\in\mathbb{N}$ and $\varepsilon>0$,
		\begin{equation}\label{4.12}
			\begin{split}
				&\bV\widehat{A}^{[m-k]}(T{M})\cap[M]\bV_{\ell^1}\\
				&\leqslant C\mathrm{Vol}(M)\cdot\Big(\mathrm{sup}_{x,x'\in \widetilde{M}}\lV I_\varepsilon(x,x')\rV_{\mathrm{End}(S_{\widetilde{M}})}\Big)\\
				&\quad\cdot\Big(\mathrm{sup}_{x\in \widetilde{M}}\int_{\widetilde{M}}\lV I_\varepsilon(x,x')\rV_{\mathrm{End}(S_{\widetilde{M}})}dv_{\widetilde{M}}(x')\Big)^k.
			\end{split}
		\end{equation}
		In particular, when $k=m$, we have $\widehat{A}^{[0]}(T{M})\cap[M]=[M]$.
	\end{prop}

	\begin{pro}
		In \eqref{3.10}, we obtain an explicit singular cycle representing the homology class $f_*(\widehat{A}^{[m-k]}(T{M})\cap[M])$. By the definition of the simplicial norm in \eqref{gromovnorm}, the sum of the absolute values of the coefficients of this cycle gives an upper bound for $\bV f_*\big(\widehat{A}^{[m-k]}(TM)\cap[M]\big)\bV_{\ell^1}$. Then by \eqref{4.1m}, this also gives an upper bound for $\bV\widehat{A}^{[m-k]}(TM)\cap[M]\bV_{\ell^1}$. Therefore,
		\begin{equation}\label{4.5n}
			\begin{split}
				&\bV\widehat{A}^{[m-k]}(T{M})\cap[M]\bV_{\ell^1}\\
				&\leqslant\sum_{\substack{\gamma_0=\gamma_{k+1}=1_\Gamma,\\ \gamma_1,\cdots,\gamma_{k}\in\Gamma,\\ \sigma\in \mathfrak{S}_{k+1}}}\Bv\mathrm{Tr}^{L^2(\gamma_{\sigma(0)}F,S_{\widetilde{M}})}\big(I_{\varepsilon,\gamma_{\sigma(k+1)},\gamma_{\sigma(k)}}\cdots I_{\varepsilon,\gamma_{\sigma(1)},\gamma_{\sigma(0)}}\big)\Bv\\
				&\leqslant C\int_{F\times \widetilde{M}^k}\lV I_\varepsilon(x_0,x_1)\cdots I_\varepsilon(x_{k},x_0)\rV_{\mathrm{End}(S_{\widetilde{M}})} dv_{\widetilde{M}^{k+1}}(x_{0},\cdots,x_{k})
				\\
				&\leqslant C  \Big(\mathrm{sup}_{x_k,x_0\in \widetilde{M}}\lV I_\varepsilon(x_k,x_0)\rV_{\mathrm{End}(S_{\widetilde{M}})}\Big)
				\\
				&\quad\cdot \int_{F\times \widetilde{M}^k}\lV I_\varepsilon(x_0,x_1)\cdots I_\varepsilon(x_{k-1},x_k)\rV_{\mathrm{End}(S_{\widetilde{M}})} dv_{\widetilde{M}^{k+1}}(x_{0},\cdots,x_{k}),
			\end{split}
		\end{equation}
		which implies \eqref{4.12}.\qed
	\end{pro}

	By the definition of $I_\varepsilon$ in \eqref{2.21} and the estimate \eqref{4.12}, this reduces to choosing suitable functions with compactly supported Fourier transforms, so the following two terms tend to zero as $\varepsilon\to\infty$,
	\begin{equation}
		\begin{split}
			\lV I_\varepsilon(x,x')\rV_{\mathrm{End}(S_{\widetilde{M}})},\quad\int_{\widetilde{M}}\lV I_\varepsilon(x,x')\rV_{\mathrm{End}(S_{\widetilde{M}})}dv_{\widetilde{M}}(x')
		\end{split}
	\end{equation}
	The pointwise estimate is relatively straightforward. The main difficulty is the integral estimate. Its form suggests using heat kernel estimates. We shall first construct an auxiliary infinite propagation kernel using Gaussian functions, and then cut off to obtain an finite propagation kernel.

	\section{Simplicial Norm Vanishing}\label{LSE}

	In this section, we prove vanishing results for simplicial norms. In \S\,\ref{lse1}, we introduce some useful heat kernel estimates. In \S\,\ref{saip}, we construct and estimate an auxiliary infinite propagation kernel. In \S\,\ref{s5.2}, we introduce the cutoff and estimate the actual finite propagation kernel.

	Again, we focus our attention on the even dimensional case. The odd dimensional case follows from rescaling the metric and taking the product with a closed odd dimensional hyperbolic manifold.

	\subsection{Heat kernel estimates}\label{lse1}
	
	Let $\widetilde\Delta$ be the (nonnegative) scalar Laplacian on $\widetilde M$. We denote its heat operator by $e^{-s\widetilde{\Delta}}$ and its heat kernel by $p_s(x,x')$.
	
	By the maximum principle of the heat equation, see for instance Dodziuk \cite{DodziukMaximumPrinciple}, for any $w\in L^2(\widetilde M)$ with  $0\leqslant w\leqslant 1$, we have
	\begin{equation}\label{eq:maxprinciple}
		0\leqslant e^{-s\widetilde\Delta}w\leqslant 1.
	\end{equation}
	Consequently, the heat kernel satisfies
	\begin{equation}\label{x5.1}
		p_s(x,x')\geqslant0,\quad \int_{\widetilde{M}}p_s(x,x')dv_{\widetilde{M}}(x')\leqslant1.
	\end{equation}
	The first inequality is immediate from the positivity preserving property. The second inequality follows from applying \eqref{eq:maxprinciple} to the charateristic function of every compact $K\subset\widetilde M$,
	\begin{equation}
		e^{-s\widetilde\Delta}(\mathbbm{1}_K)(x)=\int_{K}p_s(x,x')dv_{\widetilde{M}}(x')\leqslant1.
	\end{equation}

	The following estimate is classical, while we include a proof for completeness.
	\begin{prop}
		Under the positive scalar curvature assumption, 
		there exists $c>0$ such that
		we have	for any $s>0$,
		\begin{equation}\label{x5.3}
			\bV e^{-s\DM^2}(x,x')\bV_{\mathrm{End}(S_{\widetilde{M}})}\leqslant e^{-cs}p_s(x,x').
		\end{equation}
	\end{prop}

	\begin{pro}
		For any smooth $u\in L^2\big(\widetilde{M},S_{\widetilde{M}}\big)$, we set
		\begin{equation}
			u_s=e^{-s\widetilde D^2}u.
		\end{equation}
		By \eqref{2.18n}, we have
		\begin{equation}\label{5.6}
			\pa_su_s=-\widetilde D^2u_s=-\Big(\Delta^{S_{\widetilde{M}}}+\frac{1}{4}\mathrm{Sc}_{g^{T\widetilde{M}}}\Big)u_s.
		\end{equation}
		
		By \eqref{5.6}, the pointwise norm function $\lV u_s\rV_{S_{\widetilde{M}}}$ of $u_s$ satisfies
		\begin{equation}\label{5.7}
			\frac{1}{2}\pa_s\big(\lV u_s\rV_{S_{\widetilde{M}}}^2\big)=-\operatorname{Re}\big(\langle\Delta^{S_{\widetilde{M}}}u_s,u_s\rangle_{S_{\widetilde{M}}}\big)-\frac{1}{4}\mathrm{Sc}_{g^{T\widetilde{M}}}\lV u_s\rV_{S_{\widetilde{M}}}^2.
		\end{equation}
		We have
		\begin{equation}\label{5.8}
			\frac{1}{2}\widetilde\Delta\big(\lV u_s\rV_{S_{\widetilde{M}}}^2\big)=\operatorname{Re}(\langle\Delta^{S_{\widetilde{M}}}u_s,u_s\rangle_{S_M})-\Vert\nabla^{S_{\widetilde{M}}} u_s\Vert_{T^*\widetilde{M}\otimes S_{\widetilde{M}}}^2.
		\end{equation}
		Adding \eqref{5.7} and \eqref{5.8}, and using Kato's inequality, see for instance Hess-Schrader-Uhlenbrock \cite{MR602436}, we get
		\begin{equation}\label{5.9}
			\begin{split}
				\frac{1}{2}(\partial_s+\widetilde\Delta)\big(\lV u_s\rV_{S_{\widetilde{M}}}^2\big)=&-\Vert\nabla^{S_{\widetilde{M}}} u_s\Vert_{T^*\widetilde{M}\otimes S_{\widetilde{M}}}^2-\frac{1}{4}\mathrm{Sc}_{g^{T\widetilde{M}}}\lV u_s\rV_{S_{\widetilde{M}}}^2\\
				\leqslant&-\bV d\lV u_s\rV_{S_{\widetilde{M}}}\bV_{T^*\widetilde{M}\otimes S_{\widetilde{M}}}^2-\frac{1}{4}\mathrm{Sc}_{g^{T\widetilde{M}}}\lV u_s\rV_{S_{\widetilde{M}}}^2.
			\end{split}
		\end{equation}
		While we also have
		\begin{equation}\label{5.10}
			\begin{split}
				\frac{1}{2}(\partial_s+\widetilde\Delta)\big(\lV u_s\rV_{S_{\widetilde{M}}}^2\big)=&\lV u_s\rV_{S_{\widetilde{M}}}\cdot (\partial_s-\widetilde\Delta)\big(\lV u_s\rV_{S_{\widetilde{M}}}\big)\\
				&-\bV d\lV u_s\rV_{S_{\widetilde{M}}}\bV_{T^*\widetilde{M}\otimes S_{\widetilde{M}}}^2.
			\end{split}
		\end{equation}
		Comparing \eqref{5.9} and \eqref{5.10}, we get
		\begin{equation}
			\lV u_s\rV_{S_{\widetilde{M}}}\Big((\partial_s+\widetilde\Delta)\big(\lV u_s\rV_{S_{\widetilde{M}}}\big)+\frac{1}{4}\mathrm{Sc}_{g^{T\widetilde{M}}}\lV u_s\rV_{S_{\widetilde{M}}}\Big)\leqslant 0.
		\end{equation}
		Using approximation, in the distribution sense, we have
		\begin{equation}\label{5.12x}
			(\partial_s+\widetilde\Delta)(\lV u_s\rV_{S_{\widetilde{M}}})\leqslant-\frac{1}{4}\mathrm{Sc}_{g^{T\widetilde{M}}}\lV u_s\rV_{S_{\widetilde{M}}}\leqslant -c\lV u_s\rV_{S_{\widetilde{M}}}
		\end{equation}
		for some $c>0$ by the positive scalar curvature assumption.

		We define a function
		\begin{equation}
			w_s(x)=\displaystyle e^{cs}\lV u_s(x)\rV_{S_{\widetilde{M}}}-\int_{\widetilde M}p_s(x,x')\lV u(x')\rV_{S_{\widetilde{M}}}dv_{\widetilde{M}}(x').
		\end{equation}
		By \eqref{5.12x}, it satisfies
		\begin{equation}
			(\partial_s+\widetilde\Delta)(w_s)\leqslant 0,\quad w_0=0.
		\end{equation}
		It follows from the parabolic maximum principle that $w_s\leqslant 0$. In other words, we have for any smooth $u\in L^2\big(\widetilde{M},S_{\widetilde{M}}\big)$,
		\begin{equation}\label{5.15}
			\begin{split}
				&\bbV\int_{\widetilde M} e^{-s\widetilde D^2}(x,x')u(x')dv_{\widetilde{M}}(x')\bbV_{S_{\widetilde{M}}}\\
				&\leqslant e^{-cs}\int_{\widetilde M}p_s(x,x')\lV u(x')\rV_{S_{\widetilde{M}}}dv_{\widetilde{M}}(x').
			\end{split}
		\end{equation}

		For $x'\in\widetilde M$ and $v\in S_{\widetilde M,x'}$, choose a sequence of smooth compactly supported sections converging distributionally to $\delta_{x'}v$. Passing to the limit in \eqref{5.15} gives
		\begin{equation}
			\lV e^{-s\widetilde D^2}(x,x')v\rV_{S_{\widetilde{M}}}
			\leqslant
			e^{-cs}p_s(x,x')\lV v\rV_{S_{\widetilde{M}}}.
		\end{equation}
		Taking the supremum over all unit vectors $v$ proves \eqref{x5.3}.\qed
	\end{pro}

	\subsection{Auxiliary infinite propagation kernel}\label{saip}

	In \eqref{2.18} and \eqref{2.19}, we replace $\phi$ with auxiliary $\Phi$ and take
	\begin{equation}\label{x5.4}
		\begin{split}
			\widehat{\Phi_{\varepsilon,0}'}(\xi)&=2e^{-(\xi/\varepsilon^{\beta})^2/4},\quad \Phi_{\varepsilon,0}'(\lambda)
			=
			\frac{2\varepsilon^\beta}{\sqrt\pi}
			e^{-(\varepsilon^{\beta}\lambda)^2},\\
			\Phi_{\varepsilon,0}(\lambda)&=\int_0^\lambda\frac{2\varepsilon^\beta}{\sqrt\pi}
			e^{-(\varepsilon^{\beta}s)^2}ds,\qquad\Phi_{\varepsilon,1}=\Phi_{\varepsilon,0}^2-1.
		\end{split}
	\end{equation}
	Here $\beta>0$ is a parameter to be determined later. The reason for this choice is that the Gaussian is preserved, up to rescaling, by the Fourier transform. Replacing $\xi$ by $\xi/\varepsilon^\beta$ corresponds under Fourier inversion to replacing $\lambda$ by $\varepsilon^\beta\lambda$. Hence, as $\varepsilon\to\infty$, we have $\Phi_{\varepsilon,1}\to 0$.
	
	We introduce the complementary error function for $\lambda>0$,
	\begin{equation}\label{x5.5}
		\mathrm{erfc}(\lambda)=\int_{\lambda}^\infty\frac{2}{\sqrt{\pi}}e^{-s^2}ds=\frac{2}{\pi}\int_0^\infty
		\frac{e^{-(1+s^2)\lambda^2}}{1+s^2}ds.
	\end{equation}
	The second identity is classical. It can be verified by differentiating both sides with respect to $\lambda$. 
	
	By \eqref{x5.4} and \eqref{x5.5}, we have
	\begin{equation}\label{x5.6}
		\begin{split}
			\Phi_{\varepsilon,0}(\lambda)&=\mathrm{sgn}(\lambda)\big(1-\mathrm{erfc}(\varepsilon^\beta\lv\lambda\rv)\big),\\
			\Phi_{\varepsilon,1}(\lambda)&=-2\mathrm{erfc}(\varepsilon^\beta\lv\lambda\rv)+\mathrm{erfc}(\varepsilon^\beta\lv\lambda\rv)^2.
		\end{split}
	\end{equation}
	
	By the first and second representations in \eqref{x5.5}, using the substitutions $s=\lv\lambda\rv s^{1/2}$ and $s=\varepsilon^{2\beta}(1+s^2)$ respectively, we obtain
	\begin{equation}\label{x5.7}
		\begin{split}
			\sgn(\lambda)\mathrm{erfc}(\varepsilon^\beta\lv\lambda\rv)&=\frac{1}{\sqrt{\pi}}\int_{\varepsilon^{2\beta}}^\infty s^{-1/2}\lambda e^{-s\lambda^2}ds,\\
			\mathrm{erfc}(\varepsilon^\beta\lv\lambda\rv)&=\int_{\varepsilon^{2\beta}}^\infty
			\frac{\varepsilon^\beta}{\pi s(s-\varepsilon^{2\beta})^{1/2}}e^{-s\lambda^2}ds.
		\end{split}
	\end{equation}
	By \eqref{2.21}, \eqref{x5.6}, and \eqref{x5.7}, every entry of $I_\varepsilon$ is a polynomial of $\sgn(\lambda)\mathrm{erfc}(\varepsilon^\beta\lv\lambda\rv)$ and $\mathrm{erfc}(\varepsilon^\beta\lv\lambda\rv)$. Therefore, by \eqref{4.12}, it suffices to bound
	\begin{equation}\label{x5.8}
		\begin{split}
			&\frac{1}{\sqrt{\pi}}\int_{\varepsilon^{2\beta}}^\infty s^{-1/2}\int_{\widetilde{M}}\bV\big(\DM e^{-s\DM^2}\big)(x,x')\bV_{\mathrm{End}(S_{\widetilde{M}})}dv_{\widetilde{M}}(x')ds,\\
			&\int_{\varepsilon^{2\beta}}^\infty
			\frac{\varepsilon^\beta}{\pi s\sqrt{s-\varepsilon^{2\beta}}}\int_{\widetilde{M}}\bV e^{-s\DM^2}(x,x')\bV_{\mathrm{End}(S_{\widetilde{M}})}dv_{\widetilde{M}}(x')ds.
		\end{split}
	\end{equation}
	
	From \eqref{x5.1} and \eqref{x5.3}, the first term in \eqref{x5.8} can be bounded by
	\begin{equation}\label{5.9x}
		\begin{split}
			&\frac{1}{\sqrt{\pi}}\int_{\varepsilon^{2\beta}}^\infty s^{-1/2}\int_{\widetilde{M}^2}\bV\big(\DM e^{-\DM^2}\big)(x,x'')\bV_{\mathrm{End}(S_{\widetilde{M}})}\\
			&\qquad\qquad\qquad\qquad\cdot\bV e^{-(s-1)\DM^2}(x'',x')\bV_{\mathrm{End}(S_{\widetilde{M}})}dv_{\widetilde{M}^2}(x'',x')ds\\
			&\leqslant \frac{1}{\sqrt{\pi}}\int_{\varepsilon^{2\beta}}^\infty s^{-1/2}\int_{\widetilde{M}^2}\bV\big(\DM e^{-\DM^2}\big)(x,x'')\bV_{\mathrm{End}(S_{\widetilde{M}})}\\
			&\qquad\qquad\qquad\qquad\quad\cdot e^{-c(s-1)}p_{(s-1)}(x'',x')dv_{\widetilde{M}^2}(x'',x')ds\\
			&\leqslant\Big(\int_{\widetilde{M}}\bV\big(\DM e^{-\DM^2}\big)(x,x'')\bV_{\mathrm{End}(S_{\widetilde{M}})}dv_{\widetilde{M}}(x'')\Big)\\
			&\quad\cdot\frac{1}{\sqrt{\pi}}\int_{\varepsilon^{2\beta}}^\infty s^{-1/2}e^{-c(s-1)}ds\\
			&\leqslant C\int_{\varepsilon^{2\beta}}^\infty s^{-1/2}e^{-cs}ds=C\mathrm{erfc}(c^{1/2}\varepsilon^\beta).
		\end{split}
	\end{equation} 
	Here the last inequality follows from
	\begin{equation}
		\bV\big(\DM e^{-\DM^2}\big)(x,x'')\bV_{\mathrm{End}(S_{\widetilde{M}})}\leqslant Ce^{-c(d_{\widetilde{M}}(x,x''))^2}
	\end{equation}
	for some $C,c>0$, and the fact that $\widetilde M$ has at most exponential volume growth, see for example Chen-Wang-Xie-Yu \cite[Lemma 3.8]{MR4673948}.
	
	Using  \eqref{x5.1} and \eqref{x5.3} again, the second term of \eqref{x5.8} can be bounded by
	\begin{equation}\label{5.10x}
		\int_{\varepsilon^{2\beta}}^\infty
		\frac{\varepsilon^\beta}{\pi s(s-\varepsilon^{2\beta})^{1/2}}e^{-cs} ds=\mathrm{erfc}\big(c^{1/2}\varepsilon^\beta\big).
	\end{equation}

	For $\varepsilon$ sufficiently large, we have $\mathrm{erfc}\big(c^{1/2}\varepsilon^\beta\big)\leqslant e^{-c\varepsilon^{2\beta}}$. Hence by \eqref{5.9x} and \eqref{5.10x}, the right hand side of \eqref{4.12} tends to zero for the auxiliary infinite propagation kernel.

	\subsection{Proof of the vanishing theorem}\label{s5.2}
	
	We now prove Theorem \ref{t1.2'} by passing from the auxiliary infinite propagation kernel to the finite propagation kernel. 
	
	\begin{proof}[Proof of Theorem \ref{t1.2'}]
		Choose a smooth even function $\rho$ such that
		\begin{equation}
			0\leqslant\rho\leqslant1,
			\quad\supp(\rho)\subseteq[-1,1],\quad	\rho=1\text{ on }[-1/2,1/2].
		\end{equation}
		Recall that in the construction \eqref{2.18} and \eqref{2.19}, it is enough to define $\widehat{\phi_{\varepsilon,0}'}$ by
		\begin{equation}\label{5.12}
			\widehat{\phi_{\varepsilon,0}'}(\xi)
			=\rho\big(\tfrac{6m\xi}{\varepsilon}\big)\widehat{\Phi_{\varepsilon,0}'}(\xi)
			=2\rho\big(\tfrac{6m\xi}{\varepsilon}\big)e^{-(\xi/\varepsilon^{\beta})^2/4}.
		\end{equation}

		We claim that for any $\ell\in\mathbb{N}$, there are constants $C,c>0$ such
		that
		\begin{equation}\label{5.13x}
			\begin{split}
				\sup_{\lambda\in\R}
				(1+\lv\lambda\rv)^\ell\lv(\phi_{\varepsilon,0}-\Phi_{\varepsilon,0})(\lambda)\rv
				\leqslant
				Ce^{-c\varepsilon^{2(1-\beta)}},\\
				\sup_{\lambda\in\R}
				(1+\lv\lambda\rv)^\ell\lv(\phi_{\varepsilon,1}-\Phi_{\varepsilon,1})(\lambda)\rv
				\leqslant
				Ce^{-c\varepsilon^{2(1-\beta)}}.
			\end{split}
		\end{equation}
		To see this, by \eqref{5.12} we get
		\begin{equation}\label{5.14}
			\widehat{(\phi_{\varepsilon,0}-\Phi_{\varepsilon,0})}(\xi)
			=2\big(\rho\big(\tfrac{6m\xi}{\varepsilon}\big)-1\big)\frac{e^{-(\xi/\varepsilon^{\beta})^2/4}}{\sqrt{-1}\xi},
		\end{equation}
		which is supported on $\lv\xi\rv\geqslant C\varepsilon$.  On this set,
		\begin{equation}
			e^{-(\xi/\varepsilon^{\beta})^2/4}\leqslant e^{-c\varepsilon^{2(1-\beta)}}.
		\end{equation}
		Differentiating \eqref{5.14} finitely many times would produce a polynomial of $\varepsilon^{-1}$ and $\xi$, which can be absorbed into the Gaussian decay. We get the first inequality in \eqref{5.13x} by the inverse Fourier transform, and the second one follows from \eqref{2.19} and \eqref{x5.4}.

		By \eqref{5.9x}, \eqref{5.10x}, \eqref{5.13x}, and the Sobolev's inequality, we obtain
		\begin{equation}\label{x5.18}
			\int_{\widetilde{M}}\bV I_\varepsilon(x,x')\bV_{\mathrm{End}(S_{\widetilde{M}})}dv_{\widetilde{M}}(x')\leqslant Ce^{-c\varepsilon^{2\beta}}+Ce^{-c\varepsilon^{2(1-\beta)}+C\varepsilon}.
		\end{equation}
		Here the factor $Ce^{C\varepsilon}$ arises because $I_\varepsilon$ has propagation at most $\varepsilon/m$. Thus it should be integrated only over $B^{\widetilde{M}}(x,\varepsilon/m)$, whose volume is bounded by some $Ce^{C\varepsilon}$.

		It suffices to choose $\beta$ such that $\beta>0$ and $2(1-\beta)>1$, or equivalently $0<\beta<1/2$. Then by \eqref{x5.18}, the right hand side of \eqref{4.12} tend to zero as $\varepsilon\to\infty$. This completes the proof.
	\end{proof}

	\subsection{Obstruction to nonnegative scalar curvature}
	
	We now prove Corollary \ref{t1.2}.
	
	\begin{proof}[Proof of Corollary \ref{t1.2}]
		Suppose that $\Sc_{g^{TM}}\geqslant0$ on $M$. If $\Sc_{g^{TM}}$ is positive at some point,
		the conformal Laplacian produces a metric of positive scalar curvature,
		contrary to Theorem \ref{t1.2'}. Hence $\Sc_{g^{TM}}\equiv0$. Bourguignon's
		deformation theorem then implies that ${g^{TM}}$ is Ricci-flat, otherwise $M$ would again admit a metric of positive scalar curvature. By the Cheeger-Gromoll splitting theorem, the fundamental group $\pi_1(M)$ is virtually abelian and therefore amenable. It then follows from Gromov's mapping theorem  \cite[\S\,3.1]{MR686042} that $\bV[M]\bV_{\ell^1}=0$, which  contradicts the assumption.
		
		If $M$ is aspherical, then its universal cover is contractible, hence is spin. The first assertion
		therefore applies to $M$.
	\end{proof}

	\addcontentsline{toc}{section}{References}
	\bibliographystyle{abbrv}
	\def\cprime{$'$} \def\cprime{$'$}

\end{document}